\documentclass[pdflatex,sn-mathphys-num]{sn-jnl}
\usepackage{mathrsfs}
\usepackage{enumitem}
\usepackage{graphicx}%
\usepackage{multirow}%
\usepackage[all]{xy}
\usepackage{amsmath,amssymb,amsfonts}%
\usepackage{amsthm}%
\usepackage{mathrsfs}%
\usepackage[title]{appendix}%
\usepackage{xcolor}%
\usepackage{textcomp}%
\usepackage{manyfoot}%
\usepackage{booktabs}%
\usepackage{algorithm}%
\usepackage{algorithmicx}%
\usepackage{algpseudocode}%
\usepackage{listings}%
\usepackage{soul}
\usepackage{enumitem}
\usepackage{tikz-cd}
\theoremstyle{thmstyleone}%
\newtheorem{theorem}{Theorem}[section]%  meant for continuous numbers
\numberwithin{equation}{section}
\newtheorem{proposition}[theorem]{Proposition}% 

\theoremstyle{thmstyletwo}%
\newtheorem{remark}[theorem]{Remark}%

\theoremstyle{thmstylethree}%
\newtheorem{definition}[theorem]{Definition}%
\newtheorem{lemma}[theorem]{Lemma}
\newtheorem{corollary}[theorem]{Corollary}
\begin{document}

    \title[A Variational Principle for the Topological Pressure of Non-autonomous Iterated Function Systems on Subsets]{A Variational Principle for the Topological Pressure of Non-autonomous Iterated Function Systems on Subsets}


\author*[1]{\fnm{Yujun} \sur{Ju}} \email{yjju@ctbu.edu.cn}

\equalcont{These authors contributed equally to this work.}

\author[1]{\fnm{Lingbing} \sur{Yang}} \email{yanglingbing@ctbu.edu.cn}

\affil[1]{\orgdiv{School of Mathematics and Statistics}, \orgname{Chongqing Technology and Business University}, \orgaddress{\city{Chongqing}, \postcode{400067}, \country{People's Republic of China}}}
\equalcont{These authors contributed equally to this work.}


\abstract{
Motivated by the notion of topological entropy for free semigroup actions introduced by Bi\'s,
we define the Pesin--Pitskel topological pressure for non-autonomous iterated function systems
via the Carath\'eodory--Pesin structure.
We show that this Pesin--Pitskel topological pressure coincides with the corresponding weighted
topological pressure.
Furthermore, we establish a variational principle asserting that, for any nonempty compact subset,
the Pesin--Pitskel topological pressure equals the supremum of the associated measure-theoretic
pressures over all Borel probability measures supported on that subset.
}



\maketitle

\section{Introduction}
Entropy and pressure are fundamental quantitative tools for describing
dynamical complexity in dynamical systems.
Measure-theoretic entropy quantifies the asymptotic rate of information
production along orbits with respect to an invariant measure, whereas topological entropy reflects the global exponential
growth rate of distinguishable orbit segments.
Topological pressure extends topological entropy by incorporating potential functions and
plays a central role in thermodynamic formalism, multifractal analysis and
dimension theory.

In 1958, Kolmogorov \cite{Kolmogorov1958} associated to any measure-theoretic
dynamical system an isomorphism invariant, namely the measure-theoretic entropy
$h_{\mu}(T)$.
Later, as an invariant of topological conjugacy, Adler, Konheim and
McAndrew \cite{Adler1965} introduced the classical topological entropy
$h_{\mathrm{top}}(T)$, which may be viewed as a dynamical analogue of the upper
box dimension.
For a topological dynamical system $(X,T)$ on a compact metric space, these two
entropies are related by the variational principle
\[
h_{\mathrm{top}}(T)
=
\sup\left\{h_{\mu}(T): \mu \in M(X,T)\right\},
\]
where $M(X,T)$ denotes the set of all $T$-invariant Borel probability measures.
This variational principle was proved by Goodwyn, Dinaburg and Goodman
\cite{Goodwyn1969, Dinaburg1970, Goodman1971}.

Motivated by the analogy with Hausdorff dimension, Bowen \cite{Bowen1973}
introduced the topological entropy $h^{B}_{\mathrm{top}}(T,Z)$ for arbitrary
subsets $Z\subset X$, which is now commonly referred to as the Bowen
topological entropy. In particular, when $Z=X$, the Bowen topological entropy coincides with the
classical topological entropy, that
is,
\[
h^{B}_{\mathrm{top}}(T,X)=h_{\mathrm{top}}(T).
\]
Subsequently, Pesin and Pitskel \cite{PesinPitskel1984} generalized Bowen’s
approach and defined topological pressure for arbitrary subsets of $X$. Pesin \cite{Pesin1997} further developed a structure extending the classical
Carath\'eodory construction, now known as the Carath\'eodory--Pesin structure,
which leads to the notions of upper and lower capacity
topological pressures on arbitrary subsets, paralleling
the upper and lower box dimensions.
Feng and Huang \cite{FengHuang2012} introduced the packing topological entropy
analogous to the packing dimension and established variational principles for
the Bowen and packing topological entropies on nonempty compact sets.
Building on the work of Feng and Huang, Tang et al.\ \cite{Tang2015} and
Zhong and Cheng \cite{zhong2023variational} respectively established variational
principles for the Pesin--Pitskel topological pressure and the packing
topological pressure on nonempty compact subsets.

With the development of dynamical systems theory, increasing attention has been
devoted to systems generated by multiple transformations.
Kolyada and Snoha
\cite{kolyada1996topological} introduced the classical topological entropy and topological sup-entropy for non-autonomous dynamical systems.
Independently, Ghys, Langevin and Walczak \cite{Ghys1988} defined topological
entropy for finitely generated pseudogroups of continuous maps. The notion of topological entropy for free semigroup actions was subsequently
introduced by Bi\'s \cite{Bis2004} and Bufetov \cite{Bufetov1999}, respectively  referred to as
the maximal and averaged (topological) entropies in \cite{Huang2018topological}.
In what follows, we call them the Bi\'s and Bufetov topological entropies.
Variational principles for the Pesin--Pitskel topological pressure have been
established for non-autonomous dynamical systems
\cite{nazarian2024variational}, free semigroup actions
\cite{ZhongChen2021, xiao2022variational} and finitely generated pseudogroup
actions \cite{bis2025variational}.

Non-autonomous iterated function systems (NAIFSs for short) generalize free semigroup
actions and non-autonomous dynamical systems by allowing
the family of generating maps to vary with time.
This feature distinguishes NAIFSs from classical autonomous iterated function
systems \cite{rempe2016non}.
Consequently, many results established for free semigroup actions and
non-autonomous dynamical systems admit natural extensions to NAIFSs.
Ghane and Sarkooh \cite{ghane2019topological} introduced the notion of
topological entropy for NAIFSs via open covers, separated sets and spanning
sets.
They also introduced topological pressure for NAIFSs as an extension of the
corresponding notions for non-autonomous dynamical systems
\cite{huang2008topological} and free semigroup actions \cite{lin2018measure}.

Extending the notion of topological sup-entropy from non-autonomous dynamical
systems to equicontinuous NAIFSs, Ju and Yang \cite{ju2026factor} established a
semiconjugacy inequality for the factor map between equicontinuous NAIFSs for two
topological pressures.
They further observed that, when an equicontinuous NAIFS reduces to a free semigroup action,
the corresponding topological sup-entropy coincides with the Bi\'s topological entropy.
This identification indicates that the Bi\'s topological entropy can be viewed
as a topological sup-entropy version for free semigroup actions, whereas the Bufetov
topological entropy corresponds naturally to the classical topological entropy
of non-autonomous dynamical systems.
Accordingly, NAIFSs provide a natural bridge between non-autonomous dynamical
systems and free semigroup actions, rather than merely serving as a common generalization of these two systems.

Yang and Huang \cite{yang2025topological} introduced the topological pressure
and the capacity topological pressures for NAIFSs via the
Carath\'eodory--Pesin structure.
Together with the construction of topological pressure for NAIFSs via spanning
and separated sets in \cite{ghane2019topological}, these approaches suggest that the existing notions of topological pressure on
subsets of NAIFSs are primarily based on the averaged complexity captured by
the Bufetov topological entropy for free semigroup actions.

In contrast to the existing work based on the Bufetov (averaged) topological entropy,
to the best of our knowledge, there is currently no systematic study of
topological pressure for NAIFSs in the spirit of the Bi\'s (maximal)
topological entropy.
The aim of the present paper is to fill this gap by developing a notion of
topological pressure for NAIFSs associated with the Bi\'s topological entropy via the
Carath\'eodory--Pesin structure, and by establishing a corresponding variational principle.
Our results are inspired by the variational principles developed for
autonomous dynamical systems \cite{FengHuang2012,Tang2015},
non-autonomous dynamical systems \cite{nazarian2024variational}
and free semigroup actions \cite{ZhongChen2021,xiao2022variational}. Finally, for other variational principles concerning  topological pressures of NAIFSs, we refer the reader to \cite{cui2023variational,cui2026variational}.

The paper is organized as follows.
In Section 2, we recall basic definitions and preliminaries on non-autonomous iterated function systems and Bi\'s topological entropy for free semigroup actions.
In Section 3, we introduce the Pesin–Pitskel topological pressure and weighted topological pressure and prove their equivalence.
In Section 4, we show that the Pesin–Pitskel topological pressure can be determined by the measure-theoretic pressure of Borel probability measures, and we establish the variational principle between the Pesin–Pitskel topological pressure and the measure-theoretic pressure on nonempty compact subsets for NAIFSs.

\section{Preliminaries}

\subsection{Topological entropies of a semigroup of maps}

Let $(X,d)$ be a compact metric space, and let $G$ be a semigroup of continuous
self-maps of $X$ generated by a finite set
\(
G_1=\{f_{1},\ldots,f_{k}\}\cup\{\mathrm{id}_{X}\}.
\)
Write $G=\bigcup_{n\in\mathbb N} G_n$, where
\[
G_n=\{g_1\circ\cdots\circ g_n : g_1,\ldots,g_n\in G_1\}.
\]
Clearly, $G_m\subset G_n$ for all $m\le n$.
For each $n\in\mathbb N$, define a metric $d^{\max}_n$ on $X$ by
\[
d^{\max}_n(x,y)
=
\max\{d(g(x),g(y)) : g\in G_n\}.
\]
For $x\in X$, $n\in\mathbb N$ and $r>0$, the \((n,r)\)-Bowen ball with respect to $d^{\max}_n$ is defined by
\[
B^{\max}_n(x,r)=
\{y\in X : d^{\max}_n(x,y)<r\}.
\]
Let $\varepsilon>0$, $n\in\mathbb N$, and let $Y\subset X$.
A subset $E\subset Y$ is called $(n,\varepsilon)$\emph{-separated} if for any two distinct points $x,y\in E$,
\(
d^{\max}_n(x,y)>\varepsilon.
\)
A subset $F\subset X$ is said to $(n,\varepsilon)$\emph{-span} $Y$ if for each
$x\in Y$ there exists $y\in F$ such that
\(
d^{\max}_n(x,y)\le\varepsilon.
\)
Denote by $s(n,\varepsilon,Y)$ the maximal cardinality of an
$(n,\varepsilon)$-separated subset of $Y$, and by $r(n,\varepsilon,Y)$ the
minimal cardinality of a subset of $X$ which $(n,\varepsilon)$-spans $Y$.

Following Bi\'s \cite{Bis2004}, the \emph{Bi\'s topological entropy} of the free
semigroup action generated by $G_1$ is defined by
\[
h(G,G_1,X)
=
\lim_{\varepsilon\to0}\,
\limsup_{n\to\infty}
\frac{1}{n}\log s(n,\varepsilon,X)
=
\lim_{\varepsilon\to0}\,
\limsup_{n\to\infty}
\frac{1}{n}\log r(n,\varepsilon,X).
\]
More generally, for any subset $Y\subset X$, the Bi\'s topological entropy of
$Y$ can be defined by replacing $X$ with $Y$ in the above formula.

Ma and Wu \cite{ma2011topological} introduced a notion of topological entropy for arbitrary
subsets $Z\subset X$ with respect to $G_1$ by means of the
Carath\'eodory--Pesin structure, denoted by $h_Z(G_1)$.

Fix $r>0$.
For any $Z\subset X$, $\alpha\ge0$ and $N\in\mathbb N$, following the approach of
Pesin and Pitskel \cite{PesinPitskel1984}, define
\[
M(Z,\alpha,r,N)
=\inf\left\{
\sum_{i} \exp(-\alpha n_i)\right\},
\]
where the infimum is taken over all finite or countable collections \(\mathcal{F}=\{B^{\max}_{n_i}(x_i,r)\}_{i}\) such
that \(x_i\in X,~n_i \ge N \) and \(\mathcal{F}\) covers \(Z\), i.e., \(Z \subseteq \bigcup_{i}B^{\max}_{n_i}(x_i,r)\).
Define
\[
m(Z,\alpha,r)
=
\lim_{N\to\infty} M(Z,\alpha,r,N),
\]
and set
\[
h_Z(G_1,r)
=\inf\{\alpha : m(Z,\alpha,r)=0\}
=\sup\{\alpha : m(Z,\alpha,r)=+\infty\}.
\]
Then the topological entropy of $Z$ with respect to $G_1$ is defined by
\[
h_Z(G_1)=
\lim_{r\to0} h_Z(G_1,r).
\]
In particular, when $Z=X$, we have
\(
h_X(G_1)=h(G,G_1,X).
\)

\subsection{Non-autonomous iterated function systems}

Following~\cite{rempe2016non}, a \textit{non-autonomous iterated function system} is a pair \((X, \boldsymbol{f})\) in which \(X\) is a set and 
\(\boldsymbol{f}\) consists of a sequence \(\{\boldsymbol{f}^{(j)}\}_{j \geq 1}\) of collections of maps, where 
\[
\boldsymbol{f}^{(j)} = \{f_i^{(j)} : X \to X\}_{i \in I^{(j)}}
\]
and \(I^{(j)}\) is a non-empty finite index set for all \(j \geq 1\).  
If the set \(X\) is a compact topological space and all the maps \(f_i^{(j)}\) are continuous, 
we speak of a \textit{topological NAIFS}. 
Note that in the case where all \(f_i^{(j)}\) are contraction affine similarities, this is also referred to as a \textit{Moran set construction}. 
For simplicity, we define the following symbolic spaces for positive integers \(m, n \in \mathbb{N}\):
\[
I^{m,n} := \prod_{j=0}^{n-1} I^{(m+j)}, 
\qquad 
I^{m,\infty} := \prod_{j=m}^{\infty} I^{(j)}.
\]

A word \(w\) is called \textit{finite} if \(w \in I^{m,n}\) for some \(m,n \geq 1\); 
in this case its \textit{length} is \(n\), denoted by \(|w| := n\).  
Each word \(w \in I^{m,\infty}\) is called an \textit{infinite} word, and its length is infinite, denoted by \(|w| := \infty\).  

The time evolution of the system is defined by composing the maps \(f_i^{(j)}\) in the natural order. 
In general, for a finite (resp. infinite) word 
\[
w = w_m w_{m+1} \cdots w_{m+n-1} 
\quad (\text{resp. } w = w_m w_{m+1} \cdots)
\in I^{m,n}~(I^{m,\infty})
\]
and \(1 \leq k \leq |w|~(1 \leq k < \infty)\), we define
\[
f_w^{m,k}
:= f_{w_{m+k-1}}^{(m+k-1)} \circ \cdots \circ f_{w_{m+1}}^{(m+1)} \circ f_{w_m}^{(m)},
\qquad
f_w^{m,0} := \mathrm{id}_X.
\]

We also put \(f_w^{m,-k} := (f_w^{m,k})^{-1}\), which will be applied to sets, 
since we do not assume that the maps \(f_i^{(j)}\) are invertible.  

An NAIFS $(X,\boldsymbol{f})$ of continuous maps on a compact metric space
$(X,d)$ is said to be \emph{equicontinuous} if for every $\varepsilon>0$ there
exists $\delta>0$ such that
\[
d(x,y)<\delta
\;\Longrightarrow\;
d\bigl(f_i^{(j)}(x),\,f_i^{(j)}(y)\bigr)<\varepsilon
\]
for all $x,y\in X$, $j\ge1$, and $i\in I^{(j)}$.

Let $(X,\boldsymbol{f})$ be an equicontinuous NAIFS.
Following \cite{ju2026factor},
we define a metric
\[
d^{*}_{n}(x,y)=\sup_{i\in \mathbb{N}}\max_{w \in I^{i,n}}\max_{0\leq t \leq  n}d\left(f_{w}^{i,t}(x),f_{w}^{i,t}(y)\right).
\]
Since $\boldsymbol{f}$ is assumed to be composed of equicontinuous maps, $d_n^{*}$ is equivalent to $d$. Hence,  the metric space $(X, d_n^{*})$ is compact. A subset \(E^{*}\) of the space \(X\) is called \((n,\epsilon)^{*}\)-separated if for any two distinct points \(x,y\in E^{*}\), \(d^{*}_{n}(x,y)>\epsilon\). A set \(F^{*}\subset X\) \((n,\epsilon)^{*}\)-spans another set \(K\subset X\) provided that for each \(x\in K\) there is \(y\in F^{*}\) for which \(d^{*}_{n}(x,y)\leq\epsilon\).
We define \(s_n^{*}(\boldsymbol{f};Y;\epsilon)\) as the maximal cardinality of an \((n,\epsilon)^{*}\)-separated set in \(Y\) and \(r_n^{*}(\boldsymbol{f};Y;\epsilon)\) as the minimal cardinality of a set in \(Y\) which \((n,\epsilon)^{*}\)-spans \(Y\). Accordingly,
\[
H(\boldsymbol{f};Y)=\lim_{\varepsilon\to0}\limsup_{n\to\infty}\frac{1}{n}\log s_{n}^{*}(\boldsymbol{f};Y;\varepsilon)=\lim_{\varepsilon\to0}\limsup_{n\to\infty}\frac{1}{n}\log r_{n}^{*}(\boldsymbol{f};Y;\varepsilon).
\]
The quantity \(H(\boldsymbol{f};Y)\) is said to be the topological sup-entropy of \(\boldsymbol{f}\) on the set \(Y\).

If $\# I^{(j)}=1$ and $\boldsymbol{f}^{(j)}=\{f^{(j)}_{1}\}$
for every $j\ge1$, then \(H(\boldsymbol{f};Y)\) reduces to the topological sup-entropy \(H(f_{1,\infty};Y)\)
of an equicontinuous non-autonomous dynamical system $(X,f_{1,\infty})$, where
$f_{1,\infty}=\{f^{(j)}_{1}\}_{j=1}^{\infty}$.
Moreover, if $\boldsymbol{f}^{(i)}=\boldsymbol{f}^{(j)}$ for all $i,j\ge1$, let
$(X,G)$ be the free semigroup action generated by
\[
G_1=\{f^{(1)}_{i}: i\in I^{(1)}\}\cup\{\mathrm{id}_{X}\}.
\]
Then
\[
d^{*}_{n}(x,y)
=
\max_{w \in I^{1,n}}\max_{0\le t \le n}
d\bigl(f_{w}^{1,t}(x),f_{w}^{1,t}(y)\bigr)
=
d^{\max}_n(x,y),
\]
where $d^{\max}_n$ is the metric associated with the free semigroup action
$(X,G)$.
Consequently, $H(\boldsymbol{f};Y)$ coincides with the Bi\'s topological entropy
$h(G,G_1,Y)$.

\section{Topological pressures of NAIFSs}

\subsection{Pesin--Pitskel and lower and upper capacity topological pressures}\label{tp}
Let $(X,\boldsymbol{f})$ be a topological NAIFS on a compact metric space $(X,d)$.
Denote by $C(X,\mathbb{R})$ the Banach space of all continuous real-valued
functions on $X$ equipped with the supremum norm.
Let $\varphi \in C(X,\mathbb{R})$ be a continuous potential function for
$\boldsymbol{f}=\{\boldsymbol{f}^{(j)}\}_{j\ge1}$, where
$\boldsymbol{f}^{(j)}=\{f_i^{(j)}:X\to X\}_{i\in I^{(j)}}$.

For a finite word $w=w_1 w_2 \cdots w_n \in I^{1,n}$, define the Birkhoff sum of
$\varphi$ along $w$ by
\[
S_w \varphi(x)
:= \sum_{i=0}^{n-1}\varphi(f_{w}^{1,i}(x))=\varphi(x)
   + \varphi\big(f_{w}^{1,1}(x)\big)
   + \cdots
   + \varphi\big(f_{w}^{1,n-1}(x)\big).
\]

Considering a real number $\delta > 0$ and a point $x\in X$, we can define the $(n,\delta)$-Bowen ball at $x$ with respect to $d_n$ by
\[
B_{n}(x,\delta)=\{y \in X : d_{n}(x,y) < \delta \},
\]
where 
\[
d_{n}(x,y):=\max_{w \in I^{1,n}}\max_{0\leq t \leq  n}d\left(f_{w}^{1,t}(x),f_{w}^{1,t}(y)\right).
\]
For \(x \in X\), we write
\[
S_n\varphi(x)=\max_{w\in I^{1,n}} S_{w}\varphi(x).
\]
Also, for any \(\delta>0\), let
\[
S_n\varphi(x,\delta)=\sup_{y \in B_{n}(x,\delta)}S_{n}\varphi(y).
\]
Given $Z\subset X$, $\alpha \in \mathbb{R}$ and $N>0$, define
\begin{equation}\label{eq1}
M(Z,\boldsymbol{f},\varphi,\alpha,\delta,N)
=\inf_{\mathcal{G}}
\left\{
\sum_{B_{n_i}(x_i,\delta)\in\mathcal{G}}
\exp\!\left(
-\alpha n_i
+S_{n_i}\varphi(x_i)
\right)
\right\},
\end{equation}
where the infimum runs over all finite or countable subcollections
$\mathcal{G}=\{B_{n_i}(x_i,\delta)\}$ such that $x_i \in X$, $n_i \ge N$ and
\(
\bigcup_{i} B_{n_i}(x_i,\delta)\supset Z.
\)

We can easily verify that the function
$M(Z,\boldsymbol{f},\varphi,\alpha,\delta,N)$
is non-decreasing as $N$ increases.
Therefore, there exists the limit
\[
m(Z,\boldsymbol{f},\varphi,\alpha,\delta)
= \lim_{N\to\infty}
  M(Z,\boldsymbol{f},\varphi,\alpha,\delta,N).
\]

Moreover, we can also define
\begin{equation}\label{eq2}
R(Z,\boldsymbol{f},\varphi,\alpha,\delta,N)
=\inf_{\mathcal{G}_{N}}
\left\{
\sum_{B_{N}(x_i,\delta)\in\mathcal{G}_{N}}
\exp\!\left(
-\alpha N
+S_{N}\varphi(x_i)
\right)
\right\},
\end{equation}
where the infimum is taken over all finite or countable subcollections
$\mathcal{G}_{N}=\{B_{N}(x_i,\delta)\}$ covering $Z$, i.e.,
\(
\bigcup_{i} B_{N}(x_i,\delta)\supset Z.
\)

Let
\[
\underline{r}(Z,\boldsymbol{f},\varphi,\alpha,\delta)
= \liminf_{N\to\infty}
  R(Z,\boldsymbol{f},\varphi,\alpha,\delta,N),
\]
\[
\overline{r}(Z,\boldsymbol{f},\varphi,\alpha,\delta)
= \limsup_{N\to\infty}
  R(Z,\boldsymbol{f},\varphi,\alpha,\delta,N).
\]

By the construction of the Carath\'eodory--Pesin structure (for details, see \cite{Pesin1997}), there exist unique critical values such that the quantities 
$m(Z,\boldsymbol{f},\varphi,\alpha,\delta)$,
$\underline{r}(Z,\boldsymbol{f},\varphi,\alpha,\delta)$, and
$\overline{r}(Z,\boldsymbol{f},\varphi,\alpha,\delta)$
jump from $\infty$ to $0$.
We denote them respectively by
$P_{Z}(\boldsymbol{f},\varphi,\delta)$,
$\underline{CP}_{Z}(\boldsymbol{f},\varphi,\delta)$ and
$\overline{CP}_{Z}(\boldsymbol{f},\varphi,\delta)$.
Accordingly, we have
\[
P_{Z}(\boldsymbol{f},\varphi,\delta)
=\inf\{\alpha : m(Z,\boldsymbol{f},\varphi,\alpha,\delta)=0\}
=\sup\{\alpha : m(Z,\boldsymbol{f},\varphi,\alpha,\delta)=\infty\},
\]
\[
\underline{CP}_{Z}(\boldsymbol{f},\varphi,\delta)
=\inf\{\alpha : \underline{r}(Z,\boldsymbol{f},\varphi,\alpha,\delta)=0\}
=\sup\{\alpha : \underline{r}(Z,\boldsymbol{f},\varphi,\alpha,\delta)=\infty\},
\]
\[
\overline{CP}_{Z}(\boldsymbol{f},\varphi,\delta)
=\inf\{\alpha : \overline{r}(Z,\boldsymbol{f},\varphi,\alpha,\delta)=0\}
=\sup\{\alpha : \overline{r}(Z,\boldsymbol{f},\varphi,\alpha,\delta)=\infty\}.
\]

Since the functions 
$\delta \mapsto P_{Z}(\boldsymbol{f},\varphi,\delta)$, 
$\delta \mapsto \underline{CP}_{Z}(\boldsymbol{f},\varphi,\delta)$, 
and 
$\delta \mapsto \overline{CP}_{Z}(\boldsymbol{f},\varphi,\delta)$ 
are nonincreasing in $\delta$, their limits as $\delta \to 0$ exist.

\begin{definition}\label{4.1t}
For any set $Z\subset X$, we call the following quantities
\[
P_{Z}(\boldsymbol{f},\varphi)
:=\lim_{\delta\to0}P_{Z}(\boldsymbol{f},\varphi,\delta),
\]
\[
\underline{CP}_{Z}(\boldsymbol{f},\varphi)
:=\lim_{\delta\to0}\underline{CP}_{Z}(\boldsymbol{f},\varphi,\delta),
\]
\[
\overline{CP}_{Z}(\boldsymbol{f},\varphi)
:=\lim_{\delta\to0}\overline{CP}_{Z}(\boldsymbol{f},\varphi,\delta),
\]
the \emph{Pesin--Pitskel topological pressure} and the \emph{lower} and
\emph{upper capacity topological pressures} of $\varphi$ on the set $Z$
with respect to $\boldsymbol{f}$, respectively.
In particular, when $\varphi\equiv0$, we denote by
\(
h_{Z}(\boldsymbol{f}),
\underline{Ch}_{Z}(\boldsymbol{f}),
\overline{Ch}_{Z}(\boldsymbol{f})
\)
the \emph{Bowen topological entropy} and the \emph{lower} and
\emph{upper capacity topological entropies} of $\boldsymbol{f}$ on $Z$,
respectively.
\end{definition}

\begin{remark}
(1) It is easy to see that
\[
P_{Z}(\boldsymbol{f},\varphi)
\le
\underline{C P}_{Z}(\boldsymbol{f},\varphi)
\le
\overline{C P}_{Z}(\boldsymbol{f},\varphi).
\]

(2) If $\boldsymbol{f}^{(j)}=\{f_{0},\ldots,f_{k-1}\}$ for all $j\ge1$,
then $(X,\boldsymbol{f})$ reduces to a free semigroup action.
Moreover, when $\varphi\equiv0$, the quantity $h_{Z}(\boldsymbol{f})$
coincides with the topological entropy introduced by Ma and Wu~\cite{ma2011topological}.

(3) If $\#\boldsymbol{f}^{(j)}=1$ for all $j\ge1$, then $(X,\boldsymbol{f})$
reduces to a non-autonomous dynamical system.
In this case, the quantity $P_{Z}(\boldsymbol{f},\varphi)$
coincides with the Pesin--Pitskel topological pressure defined by
Sarkooh~\cite{nazarian2024variational}.
\end{remark}

If we replace \(S_{n_i}\varphi(x_i)\) and \(S_{N}\varphi(x_i)\) by \(S_{n_i}\varphi(x_i,\delta)\) and \(S_N\varphi(x_i,\delta)\) in  Eqs. (\ref{eq1}) and (\ref{eq2}) respectively, we can define new functions 
\(\mathcal{M}\), \(\mathcal{R}\) and \(\mathfrak{m}\),
respectively. Finally, for any $Z\subset X$, we denote the corresponding critical values by
\[
P'_{Z}(\boldsymbol{f},\varphi,\delta),\qquad
\underline{C P'}_{Z}(\boldsymbol{f},\varphi,\delta),\qquad
\overline{C P'}_{Z}(\boldsymbol{f},\varphi,\delta).
\]

\begin{theorem}\label{4.2t}
For any set $Z \subset X$, the following hold:
 \[P_{Z}(\boldsymbol{f},\varphi)=\lim\limits_{\delta\rightarrow 0}P'_{Z}(\boldsymbol{f},\varphi,\delta),\]
 \[\underline{CP}_{Z}(\boldsymbol{f},\varphi)=\lim\limits_{\delta\rightarrow 0}\underline{CP'}_{Z}(\boldsymbol{f},\varphi,\delta),\]
 \[\overline{CP}_{Z}(\boldsymbol{f},\varphi)=\lim\limits_{\delta\rightarrow 0}\overline{CP'}_{Z}(\boldsymbol{f},\varphi,\delta).\]
\end{theorem}

\begin{proof}
Our argument follows the approach of Pesin~\cite{Pesin1997}.
Fix $\delta>0$.
Since $S_n\varphi(x)\le S_n\varphi(x,\delta)$ for all $x\in X$ and $n\in\mathbb N$,
we clearly have
$P_Z(\boldsymbol f,\varphi,\delta)\le P'_Z(\boldsymbol f,\varphi,\delta)$.
Define
\[
\varepsilon(\delta)
:= \sup\Bigl\{
\bigl|\varphi(x)-\varphi(y)\bigr|
: d(x,y)<\delta \Bigr\}.
\]
Since $\varphi$ is continuous on the compact space $X$,
we have $\varepsilon(\delta)<\infty$ and
$\varepsilon(\delta)\to0$ as $\delta\to0$.
Moreover, for any $y\in B_{n}(x,\delta)$ and $w \in I^{1,n}$, we have
\[
\bigl|S_{w}\varphi(x)-S_{w}\varphi(y)\bigr|
\le n\,\varepsilon(\delta),
\]
which implies that for any \(x \in X\),
\[
S_{n}\varphi(x,\delta) \leq S_{n}\varphi(x)+n\,\varepsilon(\delta).
\]
Let $\mathcal{G}=\{B_{n_i}(x_i,\delta)\}_{i}$ be a cover of $Z$ with $x_i\in X$ and $n_i\ge N$.
Then
\begin{align*}
\mathcal{M}(Z,\boldsymbol{f},\varphi,\alpha,\delta,N)
&\le
\sum_{i}
\exp\!\left(
-\alpha n_i+S_{n_i}\varphi(x_i,\delta)
\right)\\[0.3em]
&\le
\sum_{i}
\exp\!\left(
-\alpha n_i+
S_{n_i}\varphi(x_i)+n_i\varepsilon(\delta)
\right)\\[0.3em]
&=
\sum_{i}
\exp\!\left(
-n_i(\alpha-\varepsilon(\delta))+
S_{n_i}\varphi(x_i)
\right).
\end{align*}
Thus,
\[
\mathcal{M}(Z,\boldsymbol{f},\varphi,\alpha,\delta,N)
\le
M(Z,\boldsymbol{f},\varphi,\alpha-\varepsilon(\delta),\delta,N).
\]
Letting $N\to\infty$ yields
\[
\mathfrak{m}(Z,\boldsymbol{f},\varphi,\alpha,\delta)
\le
m(Z,\boldsymbol{f},\varphi,\alpha-\varepsilon(\delta),\delta),
\]
and therefore
\[
P_{Z}(\boldsymbol{f},\varphi,\delta) \le P'_{Z}(\boldsymbol{f},\varphi,\delta)
\le
P_{Z}(\boldsymbol{f},\varphi,\delta)
+\varepsilon(\delta).
\]
Since $\varepsilon(\delta)\to0$ as $\delta\to0$, we obtain
\[
P_{Z}(\boldsymbol{f},\varphi)
=\lim_{\delta\to0}P'_{Z}(\boldsymbol{f},\varphi,\delta).
\]
The two capacity cases can be treated in the same way.
\end{proof}

The following basic properties of Pesin--Pitskel topological pressure and lower and upper capacity topological pressures of an NAIFS can be verified directly from the basic properties of the Carath\'{e}odory-Pesin dimension \cite{Pesin1997} and definitions.

\begin{proposition}\label{3.1p}
(1) $P_{\varnothing}(\boldsymbol{f},\varphi) \le 0.$

(2) $P_{Z_{1}}(\boldsymbol{f},\varphi)
      \le P_{Z_{2}}(\boldsymbol{f},\varphi)$
      if $Z_{1}\subset Z_{2}$.

(3) $P_{Z}(\boldsymbol{f},\varphi)
      =\sup\limits_{i\ge1} P_{Z_{i}}(\boldsymbol{f},\varphi)$,
      where $Z=\bigcup_{i\ge1}Z_{i}$ and each $Z_i\subset X$.
\end{proposition}

\begin{proposition}\label{3.2p}
(1)  $\underline{CP}_{\varnothing}(\boldsymbol{f},\varphi)\le0,$ 
~ $\overline{CP}_{\varnothing}(\boldsymbol{f},\varphi)\le0.$

(2) \(
\underline{CP}_{Z_{1}}(\boldsymbol{f},\varphi)
   \le \underline{CP}_{Z_{2}}(\boldsymbol{f},\varphi)\)
and
\(\overline{CP}_{Z_{1}}(\boldsymbol{f},\varphi)
   \le \overline{CP}_{Z_{2}}(\boldsymbol{f},\varphi)
\), if $Z_{1}\subset Z_{2}$.

(3) \(
\underline{CP}_{Z}(\boldsymbol{f},\varphi)
   \ge \sup_{i\ge1}\underline{CP}_{Z_{i}}(\boldsymbol{f},\varphi)\)
and
\(\overline{CP}_{Z}(\boldsymbol{f},\varphi)
   \ge \sup_{i\ge1}\overline{CP}_{Z_{i}}(\boldsymbol{f},\varphi)
\),
where $Z=\bigcup_{i\ge1}Z_{i}$ and each $Z_{i}\subset X$.

(4) If $g:X \rightarrow X$ is a homeomorphism which commutes with $\boldsymbol{f}$ (i.e., $f_i^{(j)}\circ g = g\circ f_i^{(j)}$ for all $i\in I^{(j)}$ and $j\ge1$), then
\[
P_{Z}(\boldsymbol{f},\varphi)
   = P_{g(Z)}(\boldsymbol{f},\varphi\circ g^{-1}),
\]
\[
\underline{CP}_{Z}(\boldsymbol{f},\varphi)
= \underline{CP}_{g(Z)}(\boldsymbol{f},\varphi\circ g^{-1}),\]
\[
\overline{CP}_{Z}(\boldsymbol{f},\varphi)
= \overline{CP}_{g(Z)}(\boldsymbol{f},\varphi\circ g^{-1}).
\]
\end{proposition}

\subsection{Weighted topological pressure}

For any $N \in \mathbb{N} $, $\alpha \in \mathbb{R}$, $\delta >0$, and any bounded function $h: X \rightarrow \mathbb{R}$, define
\[
W(\boldsymbol{f},\varphi, h, \alpha, \delta,N):=\inf \left\{ \sum_{B_{n_{i}}(x_{i}, \delta)} c_{i} \exp (-\alpha n_{i}+ S_{n_i} \varphi (x_i))\right\}
\]
where the infimum is taken over all finite or countable families $\{(B_{n_{i}}(x_{i},\delta),c_{i})\}$ such that $0 < c_{i} < \infty$, $x_{i} \in X$, $n_{i} \geq N$ for all $i$, and
\[
\sum_{i} c_{i} \chi_{B_{i}} \geq h
\]
where $B_{i}: =B_{n_{i}}(x_{i},\delta)$ and $\chi_{A}$ denotes the characteristic function of $A$, i.e., $\chi_{A}(x)=1$ if $x \in A$ and $0$ if $x \in X\setminus A$.

For $Z \subset X$ and $h=\chi _{Z}$, we set $W(\boldsymbol{f},\varphi, Z, \alpha, \delta,N)=W(\boldsymbol{f},\varphi, \chi _{Z}, \alpha, \delta,N)$. The quantity is nondecreasing as $N$ increases, hence the following limit exists:
\[
W(\boldsymbol{f},\varphi, Z, \alpha, \delta)= \lim_{N \rightarrow \infty}W(\boldsymbol{f},\varphi, Z, \alpha, \delta,N).
\]
Similar to the argument of Pesin  in  \cite{Pesin1997}, there exists a unique critical value $P^{W}_{Z}(\boldsymbol{f},\varphi,\delta)$ such that the quantity $W(\boldsymbol{f},\varphi, Z, \alpha, \delta)$  jumps from $\infty$ to $0$, that is,
\[
P^{W}_{Z}(\boldsymbol{f},\varphi,\delta)=\inf\{\alpha: W(\boldsymbol{f},\varphi, Z, \alpha, \delta)=0\}=\sup\{\alpha: W(\boldsymbol{f},\varphi, Z, \alpha, \delta)=\infty\}.
\]
The quantity $P^{W}_{Z}(\boldsymbol{f},\varphi,\delta)$ is nondecreasing as $\delta$ decreases, hence the following limit exists:
\[
P^{W}_{Z}(\boldsymbol{f},\varphi)=\lim_{\delta \rightarrow 0} P^{W}_{Z}(\boldsymbol{f},\varphi,\delta).
\]
We call $P^{W}_{Z}(\boldsymbol{f},\varphi)$  the \emph{weighted topological pressure} on $Z$.

Proposition \ref{p:W=B} establishes the equivalence between the Pesin–Pitskel topological pressure and the weighted topological pressure.
The proof relies on the following Vitali covering lemma (see \cite[Theorem 2.1]{Mattila1995}) and follows the approach of \cite{FengHuang2012}.

\begin{lemma}\label{5r}
Let \((X, d)\) be a compact metric space and \(\mathcal{B} = \{B(x_i, r_i)\}_{i \in \mathcal{I}}\) be a family of closed (or open) balls in \(X\). Then there exists a finite or countable subfamily \(\mathcal{B}' = \{B(x_i, r_i)\}_{i \in \mathcal{I}'}\) of pairwise disjoint balls in \(\mathcal{B}\) such that
\[
\bigcup_{B \in \mathcal{B}} B \subseteq \bigcup_{\substack{B(x_i, r_i) \in \mathcal{B}'}} B(x_i, 5r_i).
\] 
\end{lemma} 

\begin{proposition}\label{p:W=B}
Given $Z \subset X$, for any $\alpha \in \mathbb{R}$ and any $\varepsilon, \delta >0$, we have
\[
M(Z,\boldsymbol{f},\varphi,\alpha+\varepsilon,6\delta,N) \leq W(\boldsymbol{f},\varphi, Z, \alpha, \delta,N) \leq M(Z,\boldsymbol{f},\varphi,\alpha,\delta,N),
\]
when $N$ is large enough. As a result,
\[
P^{W}_{Z}(\boldsymbol{f},\varphi)=P_{Z}(\boldsymbol{f},\varphi).
\]
\end{proposition}

\begin{proof}
Let $Z \subset X$, $\alpha \in \mathbb{R}$ and $\varepsilon,\delta>0$.
Taking $h=\chi_{Z}$ and $c_{i}\equiv 1,$ we see that $W(\boldsymbol{f},\varphi, Z, \alpha, \delta,N) \leq M(Z,\boldsymbol{f},\varphi,\alpha,\delta,N)$ for each $N \in \mathbb{N}$.
In the following, we prove that
\[
  M(Z,\boldsymbol{f},\varphi,\alpha+\varepsilon,6\delta,N) \leq W(\boldsymbol{f},\varphi, Z, \alpha, \delta,N),
\]
 when $N $ is large enough.

Given any $0<\gamma<\varepsilon$, assume that $N \geq 2$ is such that $n^{2} e^{n(\gamma-\varepsilon)} \leq 1$ for $n \geq N$.
Let $\{(B_{n_{i}}(x_{i}, \delta),c_{i})\}_{i \in I}$ be a collection so that $I \subset \mathbb
{N}, 0 < c_{i} <\infty, n_{i} \geq N$ and
\[
\sum_{i} c_{i} \chi_{B_{i}} \geq \chi_{Z}
\]
where $B_{i}:=B_{n_{i}}(x_{i}, \delta)$.
In order to prove the left-hand inequality, it suffices to show that
\begin{equation}\label{inequation}
  M(Z,\boldsymbol{f},\varphi,\alpha+\varepsilon,6\delta,N) \leq \sum_{i \geq 1} c_{i} \exp \Big(-\alpha n_{i}+S_{n_i} \varphi (x_i) \Big).
\end{equation}

Denote $I_{n}: =\{ i\in I: n_{i}=n\}$ and $I_{n,m}=\{ i \in I_{n}: i \leq m\}$ for $n \geq N$ and $m \in \mathbb{N}$.
For simplicity of notation, set $B_{i}: =B_{n_{i}}(x_i, \delta)$ and $5 B_{i}=B_{n_{i}}(x_{i},5\delta)$ for $i \geq 1.$
Without loss of generality, assume that $B_{i}\neq B_{j}$ for $i \neq j$.
For $t > 0,$ set
\[
Z_{n,t}=\left\{x \in Z: \sum_{i \in I_{n}} c_{i} \chi_{B_{i}}(x) > t \right\}
\]
and
\[
Z_{n,m,t}=\left\{x \in Z: \sum_{i \in I_{n,m}} c_{i} \chi_{B_{i}}(x)>t\right\}.
\]
We divide the proof of (\ref{inequation}) into the following three steps.

Step 1. For each $n \geq N, m \in \mathbb{N}$ and $t >0$, there exists a finite set $J_{n,m,t} \subset I_{n,m}$ such that the balls $B_{i} ~(i \in J_{n,m,t})$ are pairwise disjoint, $Z_{n,m,t} \subset \bigcup_{i \in J_{n,m,t}} 5B_{i}$ and
\[
\sum_{i \in J_{n,m,t} } \exp \Big(-\alpha n+S_{n} \varphi (x_i)\Big) \leq
\frac{1}{t}  \sum_{i \in I_{n,m} } c_{i} \exp \Big(-\alpha n+S_{n} \varphi (x_i)\Big) .
\]
To prove the above inequality, we use Federer's method ({\cite[2.10.24]{Federer1969}}). Since
$I_{n,m}$ is finite, by approximating the $c_i$’s from above, we may assume that each $c_i$
is a positive rational, and then multiplying with a common denominator we may
assume that each $c_i$ is a positive integer. Let $p$ be the smallest integer satisfying $p \geq t.$
Denote $\mathcal{C}_{0}=\mathcal{B}_{0}=\{B_{i}: i \in I_{n,m}\}$, and define $u: \mathcal{B}_{0} \rightarrow \mathbb{Z}$  by $u(B_{i})=c_{i}$. We can inductively define integer-valued functions $v_{0}, v_{1}, \ldots, v_{p}$ on $\mathcal{B}_{0}$ and subfamilies  $\mathcal{B}_{1},\mathcal{B}_{2},\ldots,\mathcal{B}_{p}$ with $v_{0}=u$.
Using Lemma \ref{5r} (in which we take the metric $d_n$ instead of $d$), there exists a pairwise disjoint subfamily $\mathcal{B}_{1}$ of $\mathcal{B}_{0}$ such that $\bigcup_{B\in \mathcal{B}_{0}} B \subset \bigcup _{B\in \mathcal{B}_{1}} 5B,$ and hence $Z_{n,m,t} \subset \bigcup _{B\in \mathcal{B}_{1}} 5B.$
Define
\begin{align*}
v_{1}(B)=\begin{cases} v_{0}(B)-1,  &\text{for $B \in  \mathcal{B}_{1}$},\\
v_{0}(B), &\text{for $B \in \mathcal{B}_{0} \setminus \mathcal{B}_{1}$}.
\end{cases}
\end{align*}
Let $\mathcal{C}_1= \{ B \in \mathcal{B}_0 : v_1(B) \ge 1\}.$
Since the subfamily $\mathcal{B}_1$ is pairwise disjoint,
\[
Z_{n,m,t}
\subseteq
\left\{
x : \sum_{B\in \mathcal B_0: B \ni x} v_1(B) \ge p-1
\right\},
\]
which implies that every $x\in Z_{n,m,t}$ belongs to some ball
$B\in \mathcal B_0$ with $v_1(B)\ge 1$.
Thus
\(
Z_{n,m,t}\subset \bigcup_{B\in \mathcal C_1} B .
\)
By Lemma~\ref{5r}, we may choose a pairwise disjoint subfamily
$\mathcal B_2$ of $\mathcal C_1$ such that
\(
\bigcup_{B\in \mathcal C_1} B
\subset
\bigcup_{B\in \mathcal B_2} 5B,
\)
and hence
\(
Z_{n,m,t}\subset \bigcup_{B\in \mathcal B_2} 5B .
\)
Define
\[
v_2(B)=
\begin{cases}
v_1(B)-1, & \text{if } B\in \mathcal B_2,\\[4pt]
v_1(B),   & \text{if } B\in \mathcal B_0\setminus \mathcal B_2 .
\end{cases}
\]
Let $\mathcal{C}_2 = \{ B \in \mathcal{B}_0 : v_2(B) \ge 1 \}$. Then 
\[
Z_{n,m,t} \subset \left\{ x : \sum_{B \in \mathcal{B}_0 : B \ni x} v_2(B) \ge p - 2 \right\}
\]
which implies that every $x \in Z_{n,m,t}$ belongs to some ball $B \in \mathcal{B}_0$ with $v_2(B) \ge 1$. Thus $Z_{n,m,t} \subset \bigcup_{B \in \mathcal{C}_2} B$. Repeating this process, we get subfamilies $\mathcal{C}_{j-1}$ and disjoint subfamilies $\mathcal{B}_j$ of $\mathcal{B}_0$ such that $\mathcal{B}_j \subset \mathcal{C}_{j-1}$, $Z_{n,m,t} \subset \bigcup_{B \in \mathcal{B}_j} 5B$, $Z_{n,m,t} \subset \bigcup_{B \in \mathcal{C}_{j-1}} B$ and
\[
v_j(B) =
\begin{cases}
v_{j-1}(B) - 1, & \text{for } B \in \mathcal{B}_j, \\
v_{j-1}(B), & \text{for } B \in \mathcal{B}_0 \setminus \mathcal{B}_j.
\end{cases}
\]
Therefore,
\begin{align*}
 & \sum_{j=1}^{p} \sum_{B_{i} \in \mathcal{B}_{j}} \exp\left (-\alpha n+S_{n} \varphi (x_i)\right) \\
  &= \sum_{j=1}^{p} \sum_{B_{i} \in \mathcal{B}_{j}} (v_{j-1}(B_{i})-v_{j}(B_{i})) \exp \left(-\alpha n+S_{n} \varphi (x_i)\right) \\
  &\leq  \sum_{B_{i} \in \mathcal{B}_{0}} \sum_{j=1}^{p} (v_{j-1}(B_{i})-v_{j}(B_{i})) \exp \left(-\alpha n+S_{n} \varphi (x_i)\right) \\
  &\leq  \sum_{B_{i} \in \mathcal{B}_{0}} u(B_{i}) \exp \left(-\alpha n+S_{n} \varphi (x_i)\right) \\
  &=\sum_{i \in I_{n,m}} c_{i} \exp \left(-\alpha n+S_{n} \varphi (x_i)\right).\\
\end{align*}
Choose $j_{0} \in \{1,2,\ldots,p\}$ such that $ \sum_{B_{i} \in \mathcal{B}_{j_{0}}} \exp \left(-\alpha n+S_{n} \varphi (x_i)\right)$ is the smallest.
Then
\begin{align*}
\sum_{B_{i} \in \mathcal{B}_{j_{0}}} \exp \left(-\alpha n+S_{n} \varphi (x_i)\right)
\leq& \frac{1}{p}  \sum_{i \in I_{n,m}} c_{i} \exp \left(-\alpha n+S_{n} \varphi (x_i)\right) \\
\leq& \frac{1}{t}  \sum_{i \in I_{n,m}} c_{i} \exp \left(-\alpha n+S_{n} \varphi (x_i)\right) .
\end{align*}
Hence $J_{n,m,t}= \{i \in I_{n,m}: B_{i} \in \mathcal{B}_{j_{0}}\}$ is as desired.

Step 2. For each $n \geq N$ and $t > 0$, we have
\begin{equation}
   M(Z_{n,t},\boldsymbol{f},\varphi,\alpha+\varepsilon,6\delta,N) \leq \frac{1}{n^{2}t}\sum_{i \in I_{n}} c_{i} \exp\bigg(-\alpha n+S_{n} \varphi (x_i)\bigg).
\end{equation}
Without loss of generality, assume $Z_{n,t} \neq \emptyset$, otherwise there is nothing to prove.
Since $Z_{n,m,t} \uparrow Z_{n,t}$  ($m \rightarrow \infty$), $Z_{n,m,t}\neq \emptyset$ for all sufficiently large $m$.
Let $J_{n,m,t}$ be the sets constructed in Step 1, then $J_{n,m,t}\neq \emptyset$ when $m$ is large enough.
Define $E_{n,m,t}=\{x_{i}: i \in J_{n,m,t}\}$.
Since the space of all non-empty compact subsets of $X$ is compact with respect to the Hausdorff distance,
there exists a subsequence $\{m_{j}\}_{j\geq 1}$ of positive integers and a non-empty compact set $E_{n,t} \subset X$ such that $E_{n,m_{j},t}$ converges to $E_{n,t}$ in the Hausdorff distance as $j \rightarrow \infty$.
Since the distance of any two points in $E_{n,m,t}$ is not less than $\delta$ (with respect to $d_{n}$), so do the points in $E_{n,t}$.
Thus $E_{n,t}$ is a finite set, moreover, $\#(E_{n,m_{j},t})=\#(E_{n,t})$ for sufficiently large $j$.
Then the following holds
\[
Z_{n,m_{j},t} \subseteq  \bigcup_{i \in J_{n,m_{j},t}} 5 B_{i} = \bigcup_{x \in E_{n,m_{j},t}} B_{n}(x,5\delta)    \subseteq   \bigcup_{y \in E_{n,t}} B_{n}(y, 5.5 \delta)
\]
when $j$ is large enough, and thus $Z_{n,t} \subseteq \bigcup_{y\in E_{n,t}} B_{n} (y,6\delta) .$
By the way, $\#(E_{n,m_{j},t})=\#(E_{n,t})$ when $j$ is large enough and $\varphi$ is continuous, hence for above $\gamma >0$, when $m_j$ is large enough, each $y \in E_{n,t}$ associates to a unique $x \in E_{n,m_j,t}$ such that
\[
S_{n} \varphi(y) \leq S_{n} \varphi(x) +n\gamma.
\]
Using the result in Step 1, we have
\begin{align*}
   \sum_{y \in E_{n,t}} \exp\left (-\alpha n+S_{n} \varphi(y)\right) & \leq \sum_{x \in E_{n,m_{j},t}} \exp\left (-\alpha n+S_{n} \varphi(x)+n\gamma\right)  \\
  & \leq \frac{e^{n\gamma}}{t}  \sum_{i \in I_{n}} c_{i} \exp \left(-\alpha n+S_{n} \varphi(x_{i})\right) .
\end{align*}
 Hence
 \begin{align*}
      M(Z_{n,t},\boldsymbol{f},\varphi,\alpha+\varepsilon,6\delta,N)& \leq  \sum_{y \in E_{n,t}} \exp\left (-(\alpha+\varepsilon)n+S_{n} \varphi(y)\right) \\
   & =\frac{1}{e^{n\varepsilon}}  \sum_{y \in E_{n,t}} \exp\left (-\alpha n+S_{n} \varphi(y)\right) \\
   &\leq \frac{ e^{n\gamma}}{t  e^{n\varepsilon}}  \sum_{i \in I_{n}}  c_{i}\exp\left (-\alpha n+S_{n} \varphi(x_{i})\right) \\
    &\leq \frac{1}{t  n^{2}}  \sum_{i \in I_{n}}  c_{i}\exp\left (-\alpha n+S_{n} \varphi(x_{i})\right) .\\
 \end{align*}

Step 3. For any $t \in (0,1)$, we have
\[
M(Z,\boldsymbol{f},\varphi,\alpha+\varepsilon,6\delta,N) \leq \frac{1}{t }  \sum_{i \in I}  c_{i}\exp\left (-\alpha n_{i}+S_{n_i} \varphi(x_{i})\right) .
\]
Fix $t\in (0,1)$. Note that $\sum_{n=N}^{\infty}n^{-2} <1$ and $Z \subset \bigcup_{n=N}^{\infty} Z_{n,n^{-2}t}.$
Hence by Proposition \ref{3.1p}, we have
\begin{align*}
M(Z,\boldsymbol{f},\varphi,\alpha+\varepsilon,6\delta,N) &\leq \sum_{n=N}^{\infty}  M(Z_{n,n^{-2}t},\boldsymbol{f},\varphi,\alpha+\varepsilon,6\delta,N)  \\
& \leq \sum_{n=N}^{\infty}  \frac{1}{t}  \sum_{i \in I_{n}}  c_{i}\exp\left (-\alpha n+S_{n} \varphi(x_{i})\right) \\
&=\frac{1}{t} \sum_{i \in I}  c_{i}\exp\left (-\alpha n_{i}+S_{n} \varphi(x_{i})\right).
\end{align*}
Letting $t\to 1$ gives
\[
M(Z,\boldsymbol{f},\varphi,\alpha+\varepsilon,6\delta,N)
\le W(\boldsymbol{f},\varphi,Z,\alpha,\delta,N).
\]
\end{proof}

\section{Main Results and Proofs}
Let $M(X)$ denote the set of all Borel probability measures on $X$.
For later use, we first introduce the measure-theoretic pressure
associated with $\boldsymbol{f}$.

\begin{definition}
For $\mu \in M(X)$, define the \emph{measure-theoretic pressure} of $\mu$ for $\boldsymbol{f}$
 \[
 \underline{P}_{\mu}(\boldsymbol{f},\varphi):= \int \underline{P}_{\mu}(\boldsymbol{f},\varphi, x) d \mu(x),
 \]
 where
\[
\underline{P}_{\mu}(\boldsymbol{f},\varphi, x):=\lim_{r \rightarrow 0} \liminf_{n \rightarrow \infty} \frac{-\log \mu \big(B_{n}(x,r)\big)+S_{n} \varphi (x)}{n}.
\]
\end{definition}
\begin{remark}
When $\varphi \equiv 0$, we have the definition of measure--theoretic
entropy  of $\mu$, which we denote by
\[
\underline{h}_{\mu}(\boldsymbol{f})
:=
\int \underline{h}_{\mu}(\boldsymbol{f}, x)\, d\mu(x),
\]
where
\[
\underline{h}_{\mu}(\boldsymbol{f}, x)
:=
\lim_{r \rightarrow 0} \liminf_{n \rightarrow \infty}
-\frac{1}{n} \log \mu \bigl(B_{n}(x,r)\bigr).
\]
\end{remark}

\begin{lemma}
Let $(X,\boldsymbol f)$ be an NAIFS on a compact metric space $(X,d)$,
$\varphi \in C(X,\mathbb R)$, and $\mu \in M(X)$.
Then the function
\(
x \longmapsto \underline{P}_{\mu}(\boldsymbol f,\varphi,x)
\)
is Borel measurable and $\mu$-integrable. Consequently, the measure-theoretic pressure
$\underline{P}_{\mu}(\boldsymbol f,\varphi)$
is well defined.
In particular, the function
\(
x \longmapsto \underline h_\mu(\boldsymbol f,x)
\)
is Borel measurable and $\mu$-integrable, and hence
$\underline h_\mu(\boldsymbol f)$ is well defined.
\end{lemma}

\begin{proof}
Since the negative part of
$\underline{P}_{\mu}(\boldsymbol f,\varphi,x)$ is finite, it suffices to show that the map
\(
x \longmapsto \mu\bigl(B_n(x,\varepsilon)\bigr)
\)
is Borel measurable for every $n\in\mathbb N$ and $\varepsilon>0$.
To this end, fix $c\in\mathbb R$ and define
\[
A_c := \bigl\{ x\in X : \mu(B_n(x,\varepsilon)) \le c \bigr\}.
\]
We shall show that $A_c$ is closed.
Let $\{x_m\}_{m\ge1}\subset A_c$ be a sequence such that $x_m\to x_0$ as
$m\to\infty$.
For any $y\in B_n(x_0,\varepsilon)$, we have
\[
\max_{w \in I^{1,n}}\max_{0 \leq i \leq n}d\bigl(f^{1,i}_{w}(x_0), f^{1,i}_{w}(y)\bigr) < \varepsilon.
\]
By continuity of each composition $f^{1,i}_{w}$, it follows that
\[
d\bigl(f^{1,i}_{w}(x_m), f^{1,i}_{w}(y)\bigr)
\le
d\bigl(f^{1,i}_{w}(x_m), f^{1,i}_{w}(x_0)\bigr)
+
d\bigl(f^{1,i}_{w}(x_0), f^{1,i}_{w}(y)\bigr)
< \varepsilon
\]
for all sufficiently large $m$ and all $0\le i \le n$.
Hence,
\[
y\in \liminf_{m\to\infty} B_n(x_m,\varepsilon),
\]
which implies
\[
B_n(x_0,\varepsilon)
\subseteq
\liminf_{m\to\infty} B_n(x_m,\varepsilon)
=
\bigcup_{m=1}^{\infty}\bigcap_{k=m}^{\infty} B_n(x_k,\varepsilon).
\]
Therefore,
\[
\mu\bigl(B_n(x_0,\varepsilon)\bigr)
\le
\mu\left(
\bigcup_{m=1}^{\infty}\bigcap_{k=m}^{\infty} B_n(x_k,\varepsilon)
\right)
\le
\liminf_{m\to\infty} \mu\bigl(B_n(x_m,\varepsilon)\bigr)
\le c.
\]
This shows that $x_0\in A_c$, and hence $A_c$ is closed.
Consequently, the function $x\mapsto \mu(B_n(x,\varepsilon))$ is Borel measurable.
\end{proof}

We restate Lemma~1 in \cite{MaWen2008} below in a form analogous to the classical covering lemma.
The formulation is adapted to our setting by using dynamical balls.

\begin{lemma}\label{l:5r} 
Let $r > 0$ and $\mathcal{B}(r) = \{B_{n}(x,r): x \in X, n \in \mathbb{N}\}$. For any family $\mathcal{F}\subset \mathcal{B}(r)$, there exists a (not
necessarily countable) subfamily $\mathcal{G} \subset \mathcal{F}$ consisting of disjoint balls such that
\[
\bigcup_{B \in \mathcal{F}}B \subset \bigcup_{B_n(x,r) \in \mathcal{G}} B_n(x,3r).
\]
\end{lemma}

\begin{theorem}\label{th1}
Let $(X,\boldsymbol{f})$ be an NAIFS on a compact metric space $(X,d)$, $\mu \in M(X)$ and $Z$ be a Borel subset of $X$. For $s \in \mathbb{R}$, the following properties hold.
\begin{itemize}
    \item[(1)]
    If~ $\underline{P}_{\mu}(\boldsymbol{f},\varphi, x)\le s$~ for ~all ~$x\in Z$,  then $P_{Z}(\boldsymbol{f},\varphi)\le s$.

    \item[(2)]
    If~$\underline{P}_{\mu}(\boldsymbol{f},\varphi, x)\ge s$~ for~ all ~$x\in Z$ ~and ~$\mu(Z)>0,$ then $P_{Z}(\boldsymbol{f},\varphi) \ge s$.
\end{itemize}
\end{theorem}
\begin{proof}
We use the analogous method as that of \cite{MaWen2008}.
(1) Fix $\varepsilon> 0$, and put
\begin{align*}
 Z_m=\bigg\{ x \in Z  : &\liminf_{n\rightarrow \infty}\frac{-\log\mu (B_n(x,r))+S_{n}\varphi(x)}{n}< s+\varepsilon \\
 &~{\rm for~all}~r \in (0,1/m) \bigg\},
 \end{align*}
then we have $Z=\bigcup_{m \ge 1}Z_m$. Now fix $m \ge 1$ and $ 0<r<\frac{1}{3m} $. For each $x \in Z_m,$ there exists a strictly increasing sequence $\{n_{j}(x)\}_{j=1}^{\infty}$ such that
 \[
\log\mu  \big(B_{n_j(x)}(x,r) \big)-S_{n_{j}(x)}\varphi(x) \ge  -(s+\varepsilon) n_{j}(x)
\]
for all $j \ge 1$.
Given $N \geq 1$, set
\[
\mathcal{F}_{N}=\left\{B_{n_{j}(x)}(x,r): x \in  Z_m,  n_{j}(x)\ge N\right\}.
\]
Then $Z_{m} \subset \bigcup \limits_{B \in \mathcal{F}_{N}} B$.
By Lemma \ref{l:5r}, there exists a subfamily
$$\mathcal{G}_{N}=\{B_{n_{i}}(x_i,r)\}_{i\in I} \subset \mathcal{F}_{N}$$
consisting of disjoint balls such that
\[
Z_m \subset  \bigcup_{i \in I} B_{n_{i}}(x_i,3r),
\]
and
\[
\mu \big(B_{n_i}(x_i,r) \big)\ge \exp  \Big(-(s+\varepsilon) n_{i} +S_{n_i}\varphi(x_{i})\Big).
\]
The index set $I$ is at most countable since $\mu$ is a probability measure and $\mathcal{G}_{N}$ is a disjoint family of sets, each of
which has positive $\mu$-measure.
Therefore,
\begin{align*}
M(Z_m, \boldsymbol{f},\varphi, s+\varepsilon, 3r,N)&\le \sum_{i\in I}\exp  \bigg(-(s+\varepsilon) n_{i} +S_{n_i}\varphi(x_{i})\bigg)\\
&\le \sum_{i\in I}\mu  \big(B_{n_{i}}(x_i,r) \big)\le 1,
\end{align*}
where the disjointness of $\{B_{n_{i}}(x_i,r)\}_{i\in I}$ is used in the last inequality.
It follows that
 \[
 m(Z_m,\boldsymbol{f},\varphi, s+\varepsilon, 3r)=\lim\limits_{N\rightarrow \infty}  M(Z_m, \boldsymbol{f},\varphi, s+\varepsilon, 3r,N) \le 1,
 \]
 which in turn implies that $P_{Z_m}(\boldsymbol{f},\varphi,3r) \le s+\varepsilon$ for any $0<r<\frac{1}{3m}$.
 Taking $r \rightarrow 0$ yields
 \[
 P_{Z_m}(\boldsymbol{f},\varphi) \le s+\varepsilon~{\rm for~any}~m\ge 1.
 \]
By Proposition \ref{3.1p},
 \[
P_{Z}(\boldsymbol{f},\varphi)=P_{\cup_{m=1}^{\infty}Z_m}(\boldsymbol{f},\varphi) = \sup_{m\ge 1} \{P_{Z_m}(\boldsymbol{f},\varphi)\} \le s+\varepsilon.
 \]
Since $\varepsilon > 0$ is arbitrary, then we can get $P_{Z}(\boldsymbol{f},\varphi)\le s$.

(2) 
Fix $\varepsilon>0$ and for each $m\ge1$, set
\begin{align*}
Z_m=\Big\{ x \in Z :\ &\liminf_{n\rightarrow \infty}\frac{-\log \mu(B_{n}(x,r))
+S_{n}\varphi(x)}{n}> s-\varepsilon\\
&\qquad \text{for all } r\in(0,1/m]\Big\}.
\end{align*}
Since $\frac{-\log \mu(B_{n}(x,r))
+S_{n}\varphi(x)}{n}$ increases when $r$ decreases, it follows that
\[
Z_m=\Big\{ x \in Z :\liminf_{n\rightarrow \infty}\frac{-\log \mu(B_{n}(x,r))
+S_{n}\varphi(x)}{n}> s-\varepsilon,~r= 1/m\Big\}.
\]
Then $Z_m \subset Z_{m+1}$ and $\bigcup_{m=1}^{\infty}Z_m=Z$. So by the continuity of the measure, we have
\[
\lim_{m \to \infty}\mu(Z_m)=\mu(Z).
\]
Take $M \ge 1$ with $\mu(Z_M) > \frac12 \mu(Z)$. For every $N \ge 1$, put
\begin{align*}
Z_{M,N}
&:= \left\{ x \in Z_M : \frac{-\log \mu(B_{n}(x,r))+S_{n}\varphi(x)}{n} > s - \varepsilon \text{ for all } n \ge N \text{ and } r \in \left(0, \frac{1}{M}\right] \right\} \\
&= \left\{ x \in Z_M : \frac{-\log \mu(B_{n}(x,r))
+S_{n}\varphi(x)}{n} > s - \varepsilon \text{ for all } n \ge N \text{ and } r = \frac{1}{M} \right\}.
\end{align*}
Thus $Z_{M,N} \subset Z_{M,N+1}$ and $\bigcup_{N=1}^{\infty}Z_{M,N}=Z_{M}$.
Again, we can find $N^*\ge1$ such that
$\mu(Z_{M,N^*})>\frac12\mu(Z_{M})>0$. For every $x\in Z_{M,N^*}$, $0<r<\frac{1}{M}$, and $n\ge N^*$,
\begin{equation}\label{eq:mu-upper}
\mu(B_n(x,r))
\le \exp\Big(-(s-\varepsilon)n+S_n\varphi(x)\Big).
\end{equation}
Now fix $0<r<\frac{1}{M}$ and $N\ge N^*$, and let
\[
\mathcal{F}= \left\{B_{n_{i}}(y_i,r/2) : n_{i} \ge N \right\}
\]
be an arbitrary cover of $Z_{M,N^*}$ with $Z_{M,N^*}\cap B_{n_{i}}(y_i,r/2)\neq\emptyset$ for all $i$.
For each $i\ge 1$, choose $x_i\in Z_{M,N^*}\cap B_{n_{i}}(y_i,r/2)$.
Then by the triangle inequality,
\[
B_{n_i}(y_i,r/2)\subset B_{n_i}(x_i,r).
\]
In combination with \eqref{eq:mu-upper}, we can get
\begin{align*}
\sum_{i \ge 1}\exp\Big(-(s-\varepsilon)n_i
+S_{n_i}\varphi(y_i,r/2)\Big) &\ge
\sum_{i \ge 1}\exp\Big(-(s-\varepsilon)n_i
+S_{n_i}\varphi(x_i)\Big)\\
&\ge
\sum_{i \ge 1} \mu\bigl(B_{n_i}(x_i,r)\bigr)
\ \ge\ \mu(Z_{M,N^*})>0.
\end{align*}
Thus,
\[
\mathcal{M}(Z_{M,N^*},\boldsymbol{f},\varphi,s-\varepsilon,r/2,N)\ge \mu(Z_{M,N^*})>0
\quad\text{for all } N\ge N^*.
\]
It follows that
\[
P_Z(\boldsymbol{f},\varphi)\ge P_{Z_{M,N^*}}(\boldsymbol{f},\varphi)
\ge s-\varepsilon.
\]
Since $\varepsilon>0$ is arbitrary, we obtain $P_Z(\boldsymbol{f},\varphi)\ge s$.
\end{proof}

\begin{corollary}
Let $(X,\boldsymbol{f})$ be an NAIFS on a compact metric space $(X,d)$, $\mu \in M(X)$ and $Z$ be a Borel subset of $X$. For $s \in \mathbb{R}$, the following properties hold.
\begin{itemize}
\item[(1)]
If~ $\underline{h}_{\mu}(\boldsymbol{f}, x)\le s$~ for ~all ~$x\in Z$,  then $h_{Z}(\boldsymbol{f})\le s$.

\item[(2)]
If~$\underline{h}_{\mu}(\boldsymbol{f},x)\ge s$~ for~ all ~$x\in Z$ ~and ~$\mu(Z)>0,$ then $h_{Z}(\boldsymbol{f}) \ge s$.
\end{itemize}
\end{corollary}

The following lemma is an analogue of the approximation result of Feng and Huang, which can be viewed as a version
of dynamical classical Frostman’s lemma. We give its proof for NAIFSs, following the argument of \cite[Theorem~4.3]{FengHuang2012}.

\begin{proposition}\label{p:measure}
Let $K \subset X$ be a non-empty compact subset. Let $\alpha \in \mathbb{R}, N \in \mathbb{N} $ and $\varepsilon >0$.
Suppose that $c:=W(\boldsymbol{f},\varphi, K, \alpha, \varepsilon,N) >0$.
Then there is $\mu \in M(X)$ such that $\mu(K)=1$ and
\[
\mu(B_{n}(x,\varepsilon)) \leq \frac{1}{c} \exp \left(-\alpha n +S_{n} \varphi(x)\right), \forall x \in X,~ n \geq N.
\]
\end{proposition}

\begin{proof}
It is clear that $c < \infty.$ Define a functional $p$ on the space $C(X,\mathbb{R})$ of continuous real-valued functions on $X$ by
\[
p(g)= \frac{1}{c} W(\boldsymbol{f},\varphi, \chi_{K}\cdot g, \alpha, \varepsilon,N).
\]
Let $\mathbf{1}$ denote the constant function $\mathbf{1}(x)\equiv 1$. It is easy to verify
\begin{itemize}
    \item[(1)]
    $p(h+g) \leq p(h)+p(g)$~ for any $h, g \in C(X,\mathbb{R})$.
    \item[(2)]
    $p(t g)=tp(g)$ for any $t\geq 0$ and $g \in C(X,\mathbb{R})$.
    \item[(3)]
    $p(\mathbf{1})=1, 0\leq p(g) \leq \|g\|_{\infty}$  for any $g \in C(X,\mathbb{R})$, and $p(g)=0$  for $g \in C(X,\mathbb{R})$ with $g\leq 0$.
\end{itemize}
By the Hahn-Banach theorem, we can extend the linear functional $t \mapsto t p(\mathbf{1}), t\in \mathbb{R},$  from the subspace of the constant functions to a linear functional $L: C(X,\mathbb{R}) \rightarrow \mathbb{R}$ satisfying
\[
L(\mathbf{1})=p(\mathbf{1})=1,\  \  -p(-h) \leq L(h) \leq p(h)\  \text{for any}\  h \in C(X,\mathbb{R}).
\]
If $h \in C(X,\mathbb{R})$ with $h \geq 0$, then $p(-h)=0$ and so $L(h) \geq 0$.
Hence combining the fact that $L(\mathbf{1})=1$, we can use the Riesz representation theorem to find a Borel probability measure $\mu$ on $X$ such that $L(h)=\int h d \mu$ for $h \in C(X,\mathbb{R})$.

Next, we show that $\mu(K)=1$. To see this, for any compact set $E  \subseteq X\setminus K$, by the Urysohn lemma there is $h \in C(X,\mathbb{R})$ such that $0 \leq h\leq 1, h(x)=1$ for $x \in E$ and $h(x)=0$ for $x \in K$.
Then $h \cdot \chi_{K}\equiv 0$ and thus $p(h)=0$.
 Hence $ \mu(E) \leq L(h) \leq p(h)=0$. Since $\mu$ is regular, we have $\mu (X\setminus K)=0$. This means that $\mu(K)=1.$

We now show that
\[
\mu \bigl(B_{n}(x, \varepsilon)\bigr)
\leq
\frac{1}{c}\,
\exp \left(-\alpha n + S_{n} \varphi(x)\right),
\quad
\text{for all } x \in X \text{ and } n \geq N .
\]

To see this, for any compact set $E \subset B_{n}(x,\varepsilon)$, by Urysohn lemma, there exists $h \in C(X,\mathbb{R})$
such that $0 \leq h\leq 1, h(y)=1$ for $y \in E$ and $h(y)=0$ for $y \in X\setminus B_{n}(x,\varepsilon)$.
Then $\mu(E) \leq L(h) \leq p(h)$.
Since $h \cdot \chi_{K} \leq \chi_{B_{n}(x,\varepsilon)}$ and $n \geq N$, we have \[
W(\boldsymbol{f},\varphi, h \cdot \chi _{K}, \alpha, \varepsilon,N) \leq \exp \left(-\alpha n +S_{n} \varphi(x)\right)
\]
and thus $p(h) \leq \frac{1}{c} \exp \left(-\alpha n +S_{n} \varphi(x)\right)$.
Therefore,
$$\mu(E) \leq \frac{1}{c} \exp \left(-\alpha n +S_{n} \varphi(x)\right).$$
Using the regularity of $\mu$, we obtain
\(\mu (B_{n}(x, \varepsilon))\leq \frac{1}{c} \exp \left(-\alpha n +S_{n} \varphi(x)\right).
\)
\end{proof}

\begin{theorem}\label{th2}
Let $(X,\boldsymbol{f})$ be an NAIFS on a compact metric space $(X,d)$. If $Z \subset X$ is non-empty compact set, then
\[
P_{Z}(\boldsymbol{f},\varphi)= \sup \left\{\underline{P}_{\mu}(\boldsymbol{f},\varphi): \mu \in M(X), \mu(Z)=1\right\}.
\]
\end{theorem}
\begin{proof}
We first prove that
\[
P_{Z}(\boldsymbol{f},\varphi)
\ge
\underline{P}_{\mu}(\boldsymbol{f},\varphi)
\quad
\text{for any } \mu \in M(X) \text{ with } \mu(Z)=1 .
\]
Let $\mu$ be a given such measure.
Since $\underline{P}_{\mu}(\boldsymbol{f},\varphi)= \int \underline{P}_{\mu}(\boldsymbol{f},\varphi, x) d \mu(x)$, it follows that the set
\[
Z_{\delta}=\{x \in Z: \underline{P}_{\mu}(\boldsymbol{f},\varphi, x)  \geq \underline{P}_{\mu}(\boldsymbol{f},\varphi)-\delta\}
\]
has positive $\mu$-measure for all $\delta>0$.
Thus, by (2) of Theorem \ref{th1}, we obtain
\[
P_{Z_{\delta}}(\boldsymbol{f},\varphi) \geq \underline{P}_{\mu}(\boldsymbol{f},\varphi)-\delta.
\]
 Since $Z_{\delta} \subset Z$ for all $\delta>0$, we have $P_{Z}(\boldsymbol{f},\varphi) \geq \underline{P}_{\mu}(\boldsymbol{f},\varphi).$
 Hence,
\[
P_{Z}(\boldsymbol{f},\varphi)\geq \sup \bigg\{\underline{P}_{\mu}(\boldsymbol{f},\varphi): \mu \in M(X), \mu(Z)=1\bigg\}.
\]

We now prove the reverse inequality.
Assume that $P_{Z}(\boldsymbol{f},\varphi) > -\infty$, since otherwise there is
nothing to prove.
By Proposition \ref{p:W=B}, we have $P^{W}_{Z}(\boldsymbol{f},\varphi)=P_{Z}(\boldsymbol{f},\varphi).$
Fix any $\alpha<P^{W}_{Z}(\boldsymbol{f},\varphi)$, then there exist $\varepsilon>0$ and $N\in \mathbb{N}$ such that $c:=W(\boldsymbol{f},\varphi, Z, \alpha, \varepsilon,N) >0$.
By Proposition \ref{p:measure}, there exists $\mu \in M(X)$ with $\mu(Z)=1$ satisfying 
\[
\mu\bigl(B_{n}(x,\varepsilon)\bigr)
\le
\frac{1}{c}\,
\exp\left(-\alpha n + S_{n}\varphi(x)\right)
\quad
\text{for any } x\in X \text{ and } n\ge N .
\]
Clearly,
$$\underline{P}_{\mu}(\boldsymbol{f},\varphi, x) \geq \liminf\limits_{n \rightarrow \infty} \frac{-\log \mu (B_{n}(x,\varepsilon))+S_{n} \varphi(x)}{n}\geq \alpha$$
 for each $x \in X$ and hence $\underline{P}_{\mu}(\boldsymbol{f},\varphi)= \int \underline{P}_{\mu}(\boldsymbol{f},\varphi, x) d \mu(x)  \geq \alpha$. Since $\alpha<P_{Z}(\boldsymbol{f},\varphi)$ is arbitrary, we conclude that \[ P_{Z}(\boldsymbol{f},\varphi) \le \sup\left\{ \underline{P}_{\mu}(\boldsymbol{f},\varphi) : \mu\in M(X),\ \mu(Z)=1 \right\}. \] This completes the proof.
\end{proof}

\section*{Acknowledgements}
The authors are grateful to Dr.~Qian~Xiao for the useful comments and discussions.

\section*{Declarations}

\section*{Funding}
This work was supported by the Scientific Research Foundation of Chongqing Technology and Business University (Grant Nos. 2256003 and 2156019) and the Science and Technology Research Program of Chongqing Municipal Education Commission (Grant No. KJQN202500802).

\section*{Competing Interests}

The author declares that there are no conflicts of interest regarding this paper.

\bibliography{references}

\end{document}